\documentclass[10pt,twoside]{article}
\usepackage{mathrsfs}
\usepackage{amssymb}
\usepackage{amsmath}
\usepackage[all]{xy}
\usepackage{amsthm}
\numberwithin{equation}{section}

\newtheorem{theorem}{Theorem}[section]

\newtheorem{definition}[theorem]{Definition}
\newtheorem{remark}[theorem]{Remark}
\newtheorem{lemma}[theorem]{Lemma}
\newtheorem{example}[theorem]{Example}

\newtheorem{corollary}[theorem]{Corollary}

\newcommand{\edge}{\ar@{-}}

\newcommand{\pf}{\noindent\begin {proof}}
\newcommand{\epf}{\end{proof}}

\def\Ker{\mathop{\rm Ker}\nolimits}

\def\mod{\mathop{\rm mod}\nolimits}

\def\id{\mathop{\rm id}\nolimits}
\def\pd{\mathop{\rm pd}\nolimits}

\def\min{\mathop{\rm min}\nolimits}
\def\sup{\mathop{\rm sup}\nolimits}
\def\inf{\mathop{\rm inf}\nolimits}
\def\add{\mathop{\rm add}\nolimits}

\def\gldim{\mathop{\rm gl.dim}\nolimits}

\def\rad{\mathop{{\rm rad}}\nolimits}
\def\top{\mathop{{\rm top}}\nolimits}

\def\tridim{\mathop{\rm tri.dim}\nolimits}

\def\sup{\mathop{\rm sup}\nolimits}
\def\lim{\mathop{\underrightarrow{\rm lim}}\nolimits}

\def\mod{\mathop{\rm mod}\nolimits}

\def\id{\mathop{\rm id}\nolimits}
\def\pd{\mathop{\rm pd}\nolimits}

\def\min{\mathop{\rm min}\nolimits}
\def\sup{\mathop{\rm sup}\nolimits}
\def\inf{\mathop{\rm inf}\nolimits}
\def\add{\mathop{\rm add}\nolimits}

\def\gldim{\mathop{\rm gl.dim}\nolimits}

\def\rad{\mathop{{\rm rad}}\nolimits}

\def\top{\mathop{{\rm top}}\nolimits}

\def\sup{\mathop{\rm sup}\nolimits}
\def\lim{\mathop{\underrightarrow{\rm lim}}\nolimits}

\def\V{\mathop{\rm \mathcal{V}}\nolimits}
\def\T{\mathop{\rm \mathcal{T}}\nolimits}
\def\I{\mathop{\rm \mathcal{I}}\nolimits}

\def\LL{\mathop{\rm LL}\nolimits}

\title{ \bf The derived dimensions and syzygy finite type
\thanks{2020 Mathematics Subject Classification: 18G20, 16E10, 16E35 .}
\thanks{Keywords:
 the bounded derived category, radical layer length, syzygy finite type, derived dimension, artin algebras }}
\vspace{0.2cm}

\author { \ Junling  Zheng\thanks{Email: zhengjunling@cjlu.edu.cn}
\\
{\it \scriptsize  Department of Mathematics, China Jiliang University, Hangzhou, 310018, P. R. China
}}

\date{ }

\begin{document}

\baselineskip=16pt


\maketitle

\begin{abstract}
Let $\Lambda$ be an artin algebra, and $\mathcal{V}$ a subset of all simple modules in $\mod\Lambda$.
Suppose that $\Lambda/\rad \Lambda$ has finite syzygy type, then the
derived dimension of $\Lambda$ is at most $\ell\ell^{t_{\mathcal{V}}}(\Lambda_{\Lambda})+\pd\mathcal{V}.$
In particular,
if the global dimension of $\Lambda$ is finite, then
the derived dimension of $\Lambda$ is at most $\ell\ell^{t_{\mathcal{V}}}(\Lambda_{\Lambda})+\pd\mathcal{V}.$
This generalized the famous result which state that the derived dimension of $\Lambda$ is less than or equal to
the global dimension of $\Lambda$.
\end{abstract}
\pagestyle{myheadings}
\markboth{\rightline {\scriptsize  J. L. Zheng\emph{}}}
         {\leftline{\scriptsize
         The derived dimensions and syzygy finite type}}


\section{Introduction} 
In \cite{rouquier2006representation,rouquier2008dimensions},
Rouquier introduced the dimension $\tridim \mathcal{T}$ of
 a triangulated category $\mathcal{T}$, 
which can be
used to compute the representation dimension of 
artin algebras(\cite{rouquier2006representation,oppermann2009lower}).
Let $\Lambda$ be an artin algebra and $\mod \Lambda$ the category
of finitely generated right $\Lambda$-modules.
There is a famous result which state that 
the dimension of the bounded derived category $D^{b}(\mod \Lambda)$
is less than or equal to the global dimension of $\Lambda$(see \cite{krause2006rouquier,rouquier2008dimensions}).
In this paper, we will generalize this result.

In \cite{huard2013layer,huard2009finitistic},
 Huard, Lanzilotta and Hern\'andez introduced the notion of the radical layer
length associated with a torsion pair which is a generalization of the Loewy length.
 Let $\Lambda$ be an artin algebra and $\mathcal{V}$
a set of some simple modules in $\mod \Lambda$.
Let $t_{\mathcal{V}}$ be the torsion radical of a
torsion pair associated with $\mathcal{V}$ (see Section 3 for details). We use
$\ell\ell^{t_{\mathcal{V}}}(M)$ to
denote the $t_{\mathcal{V}}$-radical layer length of $M\in\mod\Lambda$.
For a module $M\in\mod \Lambda$, we use $\pd M$ to denote the
 projective dimensions of $M$;
in particular, set $\pd M=-1$ if $M=0$.
 For a subclass $\mathcal{B}$ of $\mod \Lambda$,
  the {\bf projective dimension} $\pd\mathcal{B}$
of $\mathcal{B}$ is defined as
\begin{equation*}
\pd \mathcal{B}=
\begin{cases}
\sup\{\pd M\;|\; M\in \mathcal{B}\}, & \text{if} \;\; \mathcal{B}\neq \varnothing;\\
-1,&\text{if} \;\; \mathcal{B}=\varnothing.
\end{cases}
\end{equation*}

Now, let us list some results about the upper bound of
the dimension of bounded derived categries.
\begin{theorem} \label{thm1.1}
Let $\Lambda$ be an artin algebra and $\mathcal{V}$
a set of some simple modules in $\mod \Lambda$. We have
\begin{enumerate}
\item[$(1)$] {\rm (\cite[Proposition 7.37]{rouquier2008dimensions})}
 $\tridim D^{b}(\mod \Lambda) \leqslant \LL(\Lambda)-1;$
\item[$(2)$] {\rm (\cite[Proposition 7.4]{rouquier2008dimensions}
and \cite[Proposition 2.6]{krause2006rouquier})}
 $\tridim D^{b}(\mod \Lambda) \leqslant \gldim \Lambda;$
\item[$(3)$] {\rm (\cite[Theorem 3.8]{zheng2020upper})}
  $\tridim D^{b}(\mod \Lambda) \leqslant
  (\pd\mathcal{V}+2)(\ell\ell^{t_{\mathcal{V}}}(\Lambda)+1)-2;$
\item[$(4)$] {\rm (\cite{zheng2020thedimension})}
$\tridim D^{b}(\mod \Lambda) \leqslant
2(\pd\mathcal{V}+\ell\ell^{t_{\mathcal{V}}}(\Lambda))+1;$

\item[$(5)$] {\rm (\cite{zheng2021radicalfinite})} if $\ell\ell^{t_{\V}}(\Lambda_{\Lambda})\leqslant 2$,
$\tridim D^{b}(\mod \Lambda) \leqslant
\pd\mathcal{V}+3;$

\item[$(6)$] {\rm (\cite{zheng2021mnIIalgebra})} if $\ell\ell^{t_{\V}}(\Lambda_{\Lambda})\geqslant 2$,
$\tridim D^{b}(\mod \Lambda) \leqslant
2\ell\ell^{t_{\V}}(\Lambda_{\Lambda})+\pd\mathcal{V}-1.$
\end{enumerate}
\end{theorem}

\begin{theorem}
Let $\Lambda$ be an artin algebra, and $\mathcal{V}$ a subset of all simple modules in $\mod\Lambda$.
Suppose that $\Lambda/\rad \Lambda$ has finite syzygy type, then
$\tridim D^{b}(\mod \Lambda)\leqslant \ell\ell^{t_{\mathcal{V}}}(\Lambda_{\Lambda})+\pd\mathcal{V}.$
\end{theorem}

    We also give an example to explain
     our results. Sometimes,
     we may be able to get a better upper
     bound for the dimension of
   the bounded derived category of $\Lambda$.

  \section{Preliminaries}
  \subsection{Some notions}
  Let us recall some basic notions in \cite{Wei2017Derived}.
  Let $\mathcal{C}\subseteq \mod\Lambda$. A complex $X$ is a set
  $\{X_{i}\;|\;i\in \mathbb{Z}\}$
  equipped with the following set of
   homomorphisms $$\{d_{i}^{X}: X_{i}\longrightarrow X_{i-1}\;|\;d_{i}^{X}\circ d_{i+1}^{X}=0, i\in \mathbb{Z}\}.$$
 For convinence, we can write $X=\{X_{i}, d_{i}^{X}\}$.
  A complex $X=\{X_{i},d_{i}^{X}\}$ is said to be right(respectively, left) bounded if $X_{i}=0$ for all but finitely
   many negative (respectively, positive) integers $i$.
  A complex $X$ is said to be bounded if $X_{i}=0$ for all but finitely many $i$, and is said to be
  homologically bounded if  $H_{i}(X)=0$ for all but finitely many $i$.

 Let $I\subseteq \mathbb{Z}$. If the homologies $H_{i}(X)=0$(respectively, $X_{i}=0$) for all $i\notin I$,
 then we say the complex $X$ has a homological (respectively, representation) support $I$.
 In particular, we identify a module $M$ with a complex $\{M_{i},d_{i}=0\}$ concentred on the $0$th term, that is,
$M_{0}=M$ and $M_{i}=0$ for all $i\neq 0$.
We denote by $\mathscr{D}^{I}(\mod\Lambda)$ the subcategory of the derived category $\mathscr{D}(\mod\Lambda)$
over $\mod\Lambda$ consisting of complexes with homological support $I$.
We denote $\add M$ by the subcategory of $\mod\Lambda$ consisting of direct summands of finite direct sums of module $M$.
It is well known that
$\mathscr{K}^{-,b}(\add\Lambda)$ is equivalent to $\mathscr{D}^{b}(\mod\Lambda)$ as triangulated categories.
For a complex $X\in \mathscr{D}^{b}(\mod\Lambda)$, a projective resolution of $X$ is a complex $P\in \mathscr{K}^{-,b}(\add\Lambda)$
such that $P\cong X$ in $\mathscr{D}(\mod\Lambda)$. We denote by $\sigma_{I}(X)$
the brutal truncated complex which is obtained from the complex $X$
by replacing each $X_{i}$, where $i\notin I$, with $0$.

\begin{definition}{\rm (\cite[Definition 3.1]{Wei2017Derived})
Let complex $X\in\mathscr{D}^{b}(\mod\Lambda)$ and $n\in \mathbb{Z}$. Let
$P$ be a projective resolution of complex $X$. We say that a complex in $\mathscr{D}^{b}(\mod\Lambda)$
is an $n$th syzygy of $X$, if it is isomorphic to $(\sigma_{[n,+\infty)}(P))[-n]$ in $\mathscr{D}(\mod\Lambda)$.
The $n$th syzygy of $X$ is denoted by $\Omega_{\mathscr{D}}^{n}(X)$.
}
\end{definition}

Let us recall some basic properties on the syzygy of complexes.
\begin{lemma}{\rm (see \cite[Lemma 3.3]{Wei2017Derived})}\label{lemwei1}
Let $M,N\in D^{b}(\mod\Lambda)$ and $n$ be an integer. Then
 $\Omega_{\mathscr{D}}^{n}(M\oplus N)\cong \Omega_{\mathscr{D}}^{n}(M)\oplus \Omega_{\mathscr{D}}^{n}(N).$
\end{lemma}

\begin{lemma}{\rm (see \cite[Lemma 3.4]{Wei2017Derived})}\label{lemwei2}
Let $M\in D^{b}(\mod\Lambda)$. Then for any integer $n$, there is a triangle
$$\Omega_{\mathscr{D}}^{n+1}(M)\longrightarrow Q \longrightarrow\Omega_{\mathscr{D}}^{n}(M)\longrightarrow(\Omega_{\mathscr{D}}^{n+1}(M))[1] $$
where $Q$ is projective.
\end{lemma}

\begin{lemma}{\rm (see \cite[Proposition 3.8]{Wei2017Derived})}\label{lemwei3}
Let $L\longrightarrow M \longrightarrow N \longrightarrow L[1]$ be a triangle in $D^{b}(\mod\Lambda).$
Then for each integer $n$, we have a triangle
$$\Omega_{\mathscr{D}}^{n}(L)\longrightarrow \Omega_{\mathscr{D}}^{n}(M) \longrightarrow \Omega_{\mathscr{D}}^{n}(N) \longrightarrow (\Omega_{\mathscr{D}}^{n}(L))[1].$$
\end{lemma}

  \subsection{The dimension of triangulated category}
  We recall some notions from
  \cite{rouquier2006representation,rouquier2008dimensions,
  oppermann2009lower}.
 Fix subcategories $\I,\I_{1},\I_{2}$ of a triangulated categories.
  Let $\langle \I \rangle_{1}$ be the smallest full subcategory
  of $\T$ which contains $\I$ and is closed under taking
  finite direct sums, direct summands, and all shifts.
  objects in $\I$.
  We denote $\I_{1}*\I_{2}$
  by the full subcategory of all extensions between them, that is,
  $$\I_{1}*\I_{2}=\{ X\mid  X_{1} \longrightarrow X
  \longrightarrow X_{2}\longrightarrow X_{1}[1]\;
  {\rm with}\; X_{1}\in \I_{1}\; {\rm and}\; X_{2}\in \I_{2}\}.$$
Let $\I_{1}\diamond\I_{2}:=\langle\I_{1}*\I_{2} \rangle_{1}$, and one can define
  \begin{align*}
  \langle \I \rangle_{0}:=0,\;
  \langle \I \rangle_{n+1}:=\langle \I
  \rangle_{n}\diamond\langle \I \rangle_{1}\;{\rm for\; any \;}
  n\geqslant 1.
  \end{align*}
By the octahedral axiom, one can know that $*$ and $\diamond$ are associative.

  \begin{definition}{\rm
    (\cite[Definiton 3.2]{rouquier2006representation})\label{tri.dimenson2.1}
  The {\bf dimension} $\tridim \T$ of a triangulated category $\T$
  is the minimal $d$ such that there exists an object $M\in \T$ with
  $\T=\langle M \rangle_{d+1}$. If no such $M$ exists for any $d$,
  then we set $\tridim \T=\infty.$
  }
  \end{definition}
The dimension $\tridim D^{b}(\mod \Lambda)$ also is said to be the derived dimension of $\Lambda$,
and sometimes is denoted by $\text{der.dim} (\Lambda)$(see \cite{chen2008algebras}).
  \begin{lemma}{\rm (\cite[Lemma 7.3]{psaroudakis2014homological})}\label{lem2.5}
Let $X, Y$ be
   two objects of a triangulated category $\T$.
  Then for each $m,n \geqslant 0$, we have
  $\langle X \rangle _{m}\diamond \langle Y \rangle _{n}
  \subseteq \langle X\oplus Y \rangle _{m+n}.$

  \end{lemma}
\section{Main results}
\subsection{An answer of Wei's problem}\label{weiproblem}
In \cite{Wei2017Derived}, Wei give the following

$\mathbf{Problem}$ Is every syzygy-finite algebra derived to an algebra of finite representation type?
\\
The answer is negative.
Let $\Lambda$ be the Beilinson algebra $kQ/I$ with (see \cite[Example 3.7]{oppermann2010representation1})
$$\xymatrix{
&0 \ar@/_1pc/[r]_{x_{n}}\ar@/^1pc/[r]^{x_{0}}_{\vdots}
&1\ar@/_1pc/[r]_{x_{n}}\ar@/^1pc/[r]^{x_{0}}_{\vdots}
&2\ar@/_1pc/[r]_{x_{n}}\ar@/^1pc/[r]^{x_{0}}_{\vdots}
&3&\cdots &n-1\ar@/_1pc/[r]_{x_{n}}\ar@/^1pc/[r]^{x_{0}}_{\vdots}&n
}$$
$I=(x_{i}x_{j}-x_{j}x_{i})$(where $0 \leqslant i, j \leqslant n$ and $n\geqslant 2$).
We know that $\tridim D^{b}(\mod \Lambda)=\gldim \Lambda =n $.
Since $\gldim \Lambda =n $ is finite, we know that $\Lambda$
is syzygy-finite algebra.
If there exists a finite representation algebra $A$ such that $D^{b}(\mod A)$ is equivalent to $D^{b}(\mod\Lambda)$, we know that
$\tridim D^{b}(\mod A)=\tridim D^{b}(\mod\Lambda)\geqslant 2$ by \cite[Lemma 3.4]{rouquier2008dimensions}.
On the other hand, the derived dimension of every finite representation algebra is less or equal to $1$(see \cite{han2009derived}),
then we know that $\tridim D^{b}(\mod A)\leqslant 1$, contradiction!

\subsection{The derived dimension of algebra $\Lambda$ with $\top \Lambda$ finite syzygy type}\label{mainresult}
For a module $M\in \mod\Lambda$, we use $\rad M$ and $\top M$
to denote the radical and top of $M$ respectively.
We use $\add M$ to denote the subcategory of $\mod\Lambda$
consisting of direct summands of finite direct sums of module $M$.
Let $\mathcal{V}$ be a subset of all simple modules, and $\mathcal{V}'$
the set of all the others simple modules in $\mod\Lambda$.
We write $\mathfrak{F}(\mathcal{V}):=\{M\in\mod\Lambda\;|\;\text{ there exists a chain }
0\subseteq M_{0}\subseteq M_{1}\subseteq M_{2}\subseteq \cdots\subseteq M_{m-1}\subseteq M_{m}=M
\text{ of submodules of  } M \text{ such that each quotients } M_{i}/M_{i-1}\in\mathcal{V}\}.$
$\mathfrak{F}(\mathcal{V})$ is closed under extensions, submodules and quotients modules. Then we have
a torsion part $(\mathcal{T},\mathfrak{F}(\mathcal{V}))$, and the corresponding torsion radical is
denoted by $t_{\mathcal{V}}$, and we set $q_{t_{\mathcal{V}}}(M)=M/t_{\mathcal{V}}(M)$ for each $M\in \mod\Lambda$.
Note that, $\rad,\top,t_{\mathcal{V}}$ and $q_{t_{\mathcal{V}}}$ are covariant additive functors.

\begin{remark}\label{rem-gldim}
{\rm (see \cite[Remark 3.16(2)]{zheng2020upper})
 If $\mathcal{V}$ is the set of all simple moduels, then the torsion pair $(\mathcal{T}_{\mathcal{V}}, \mathfrak{F}(\mathcal{V}))=(0,\mod \Lambda)$, and
$\ell\ell^{t_{\mathcal{V}}}(\Lambda)=0$ and $\pd \mathcal{V}=\gldim \Lambda$.
In this case, $\pd \mathcal{V}+\ell\ell^{t_{\mathcal{V}}}(\Lambda)=\gldim \Lambda$.}
\end{remark}

\begin{definition}{\rm (\cite{huard2013layer})
The $t_{\mathcal{V}}$-radical layer length is a function
$\ell\ell^{t_{\mathcal{V}}}:\;\;\mod\Lambda \longrightarrow \mathbb{N}\cup \{\infty\}$
 via $$\ell\ell^{t_{\mathcal{V}}}(M)=\inf\{i\geqslant 0\;|\;t_{\mathcal{V}}\circ F_{t_{\mathcal{V}}}^{i}(M)=0, M\in \mod\Lambda\}$$
 where $F_{t_{\mathcal{V}}}=\rad\circ t_{\mathcal{V}}. $
 }
\end{definition}

Let $M\in\mod\Lambda$. Let $f:P\longrightarrow M$ be the
projective cover of $M$, and we set $\Omega^{1}(M):=\Ker f$.
Inductively, $\Omega^{i}(M):=\Omega^{1}(\Omega^{i-1}(M))$
for each $i\geqslant 1$, where $\Omega^{0}(M)=M$ and $\Omega(M):=\Omega^{1}(M).$
\begin{definition}\label{syzygy-type}
{\rm (\cite{Goodearl1998repetitive})
  Let $M\in \mod \Lambda$.
  We say that $M$ has finite syzygy type if
  there is a module $V\in \mod \Lambda$ such that,
  for each $i\geqslant 0$, $\Omega^{i}(M)\in \add V.$
  }
\end{definition}

\begin{lemma}\label{lem3.5}
Let $X$ be a bounded complex and $X\in D^{b}(\mod \Lambda)$
and all $X_{i}$ semisimple.
Then

$(1)$ $X\cong \oplus_{i}H_{i}(X)[i]$ and $X\in \langle \Lambda/\rad\Lambda \rangle_{1}$
in $D^{b}(\mod \Lambda)$.

$(2)$ if $\Lambda/\rad \Lambda$ has finite syzygy type, then there is a module $V$, such that
for each integer $m$, we have
$\langle \Omega^{m}_{\mathscr{D}}(X) \rangle_{1}\subseteq\langle  V\rangle_{1}.$
\end{lemma}
\begin{proof} $(1)$ See \cite[Lemma 3.5]{zheng2020upper}.

$(2)$
By assumption and Definition \ref{syzygy-type},
  there is a module $V\in \mod \Lambda$ such that,
  for each $i\geqslant 0$, $\Omega^{i}(\Lambda/\rad \Lambda)\in \add V.$
  Note that, for each integer $i$, we know that
  $H_{i}(X)$ is semisimple,
\begin{align*}
   \langle \Omega^{m}_{\mathscr{D}}(X) \rangle_{1}&=\langle\Omega^{m}_{\mathscr{D}}(\oplus_{i}H_{i}(X)[i])  \rangle_{1}\;\;\;\;\text{(by (1))}\\
   &=\langle\oplus_{i}\Omega^{m}_{\mathscr{D}}(H_{i}(X)[i])  \rangle_{1}\;\;\;\;\text{(by Lemma \ref{lemwei1})}\\
    &\subseteq\langle\oplus_{i}\Omega^{m}_{\mathscr{D}}((\Lambda/\rad\Lambda)[i])  \rangle_{1}\;\;\;\;(H_{i}(X)\text{ is semisimple})\\
                &\subseteq\langle V \rangle_{1}.
    \end{align*}
\end{proof}
\begin{lemma}{\rm (\cite[Lemma 3.7]{zheng2020upper})}\label{lem-finite-step}
Let $\mathcal{V}$ be a subset of $\mathcal{S}^{<\infty}$. For a bounded complex $X\in D^{b}(\mod \Lambda)$.
If $\ell\ell^{t_{\mathcal{V}}}(\Lambda_{\Lambda})=n$, then $F^{n}_{t_{\mathcal{V}}}(X)\in\langle\Lambda\rangle_{\pd \mathcal{V}+1} $.
\end{lemma}

It is well known that $\mathscr{D}^{b}(\mod\Lambda)$ is equivalent to $\mathscr{K}^{-,b}(\add \Lambda)$ as
triangulated categories. A projective resolution of complex $M$ is a complex $P\in\mathscr{K}^{-,b}(\add \Lambda)$
such that $P\cong M$ in $\mathscr{D}(\mod \Lambda)$.
We need the following observation in this paper.
\begin{lemma}\label{lem-syzygy-zero}
If the following bounded complex
\[\xymatrix{
X: \cdots   \ar[rr]^{d_{i+2}} & & X_{i+1}\ar[rr]^{d_{i+1}}    & & X_{i}\ar[rr]^{d_{i}}
& & X_{i-1}\ar[rr]^{d_{i-1}}   &&\cdots   &}\]
with all $X_{i}$ in $\mod \Lambda$ satisfied the following two conditions

$(1)$ $X\in \mathscr{D}^{(-\infty,-1]}(\mod \Lambda);$

$(2)$ there exists an integer $m$ such that $\pd X_{i}\leqslant m<\infty$ and $\pd \Ker d_{i}\leqslant m<\infty$
for each $i$.\\
Then there is a complex $P\in\mathscr{K}^{b}(\add \Lambda)$ such that
$P\cong X$ in $\mathscr{D}(\mod \Lambda)$.
In particular,
$\Omega_{\mathscr{D}}^{m+1}(X)=0$ in $\mathscr{D}(\mod\Lambda)$.
\end{lemma}

\begin{lemma}\label{lem3.6}
Let $\mathcal{V}$ be a subset of $\mathcal{S}^{<\infty}$. Given the following bounded complex
 $Y\in \mathscr{D}^{(-\infty,-1]}(\mod \Lambda)$
\[\xymatrix{
Y: \cdots   \ar[rr]^{d_{i+2}} & & Y_{i+1}\ar[rr]^{d_{i+1}}    & & Y_{i}\ar[rr]^{d_{i}}
& & Y_{i-1}\ar[rr]^{d_{i-1}}   &&\cdots   &}\]
with all $Y_{i}\in \mathfrak{F}(\mathcal{V})$, we can get
 $\Omega^{\pd\mathcal{V}+1}_{\mathscr{D}}(Y)=0$ in $D^{b}(\mod\Lambda)$.
 In particular,  for each bounded complex $X\in \mathscr{D}^{(-\infty,-1]}(\mod \Lambda)$,
we have that $\Omega_{\mathscr{D}}^{\pd\mathcal{V}+1}(q_{t_{\mathcal{V}}}(X))=0$ in $D^{b}(\mod\Lambda)$.
\end{lemma}
\begin{proof}
  Similar to the proof of \cite[Lemma 3.6]{zheng2020upper}, we can get
  $\pd\Ker d_{i}\leqslant \pd \mathcal{V}<\infty $ and $\pd H_{i}(Y)\leqslant \pd \mathcal{V}<\infty $ for all $i$.
\end{proof}
%
%

\begin{theorem}\label{maintheorem}
Let $\Lambda$ be an artin algebra, and $\mathcal{V}$ a subset of all simple modules in $\mod\Lambda$.
Suppose that $\Lambda/\rad \Lambda$ has finite syzygy type, then
$\tridim D^{b}(\mod \Lambda)\leqslant \ell\ell^{t_{\mathcal{V}}}(\Lambda_{\Lambda})+\pd\mathcal{V}.$
\end{theorem}
\begin{proof}
Two special cases are $\mathcal{V}=\varnothing$ and $\ell\ell^{t_{\mathcal{V}}}(\Lambda_{\Lambda})=0$, see the proof of \cite[Theorem 3.8]{zheng2020upper}.
Now we will consider the case $\ell\ell^{t_{\mathcal{V}}}(\Lambda_{\Lambda})=n\geqslant 1$ and $\pd \mathcal{V}=m<\infty$.
Since $\Lambda/\rad \Lambda$ has finite syzygy type, by Lemma \ref{lem3.5}, we can set $V\in \mod \Lambda$, such that
\begin{gather}\label{triangle0}
  \xymatrix{
 \langle\Omega_{\mathscr{D}}^{m+1}(\Lambda/\rad \Lambda)  \rangle_{1}\subseteq\langle\Omega_{\mathscr{D}}^{m+1}(V)  \rangle_{1}.
  }
\end{gather}

Taking a bounded complex $X\in \mathscr{D}^{(-\infty,-1]}$, by the following short exact sequences of
complexes
$$0\longrightarrow t_{\mathcal{V}}(X) \longrightarrow X\longrightarrow q_{t_{\mathcal{V}}}(X)\longrightarrow 0$$
and
$$0\longrightarrow F_{t_{\mathcal{V}}}(X) \longrightarrow t_{\mathcal{V}}(X)\longrightarrow \top t_{\mathcal{V}}(X)\longrightarrow 0,$$
we can get the following two triangles in $D^{b}(\mod\Lambda)$
\begin{gather}\label{triangle1}
  \xymatrix{
 t_{\mathcal{V}}(X) \longrightarrow X\longrightarrow q_{t_{\mathcal{V}}}(X)\longrightarrow t_{\mathcal{V}}(X)[1]
  }
\end{gather}
and
\begin{gather}\label{triangle2}
  \xymatrix{
 F_{t_{\mathcal{V}}}(X) \longrightarrow t_{\mathcal{V}}(X)\longrightarrow \top t_{\mathcal{V}}(X)\longrightarrow F_{t_{\mathcal{V}}}(X)[1].
  }
\end{gather}
Moreover, by Lemma \ref{lemwei2} and (\ref{triangle1}) and (\ref{triangle2}), we can obtain the following two triangles
\begin{gather}\label{triangle3}
  \xymatrix{
 \Omega_{\mathscr{D}}^{\pd\mathcal{V}+1}(t_{\mathcal{V}}(X)) \longrightarrow \Omega_{\mathscr{D}}^{m+1}(X)\longrightarrow \Omega_{\mathscr{D}}^{m+1}(q_{t_{\mathcal{V}}}(X))\longrightarrow \Omega_{\mathscr{D}}^{m+1}(t_{\mathcal{V}}(X)[1])
  }
\end{gather}
and
\begin{gather}\label{triangle4}
  \xymatrix{
 \Omega_{\mathscr{D}}^{m+1}(F_{t_{\mathcal{V}}}(X)) \longrightarrow \Omega_{\mathscr{D}}^{m+1}(t_{\mathcal{V}}(X))\longrightarrow
 \Omega_{\mathscr{D}}^{m+1}(\top t_{\mathcal{V}}(X))\longrightarrow \Omega_{\mathscr{D}}^{m+1}(F_{t_{\mathcal{V}}}(X)[1]).
  }
\end{gather}
By Lemma \ref{lem-syzygy-zero}, we have $\Omega_{\mathscr{D}}^{m+1}(q_{t_{\mathcal{V}}}(X))=0$ in $D^{b}(\mod\Lambda)$.
By the triangle (\ref{triangle3}), we have
\begin{gather}\label{triangle5}
  \xymatrix{
 \Omega_{\mathscr{D}}^{m+1}(t_{\mathcal{V}}(X)) \cong\Omega_{\mathscr{D}}^{m+1}(X).
  }
\end{gather}
And then we can get
 \begin{align*}
   \langle \Omega_{\mathscr{D}}^{m+1}(X) \rangle_{1}&=\langle\Omega_{\mathscr{D}}^{m+1}(t_{\mathcal{V}}(X)) \rangle_{1}\;\;\;\;\text{(by (\ref{triangle5}))}\\
                &\subseteq\langle\Omega_{\mathscr{D}}^{m+1}(F_{t_{\mathcal{V}}}(X))  \rangle_{1}\diamond \langle\Omega_{\mathscr{D}}^{m+1}(\top t_{\mathcal{V}}(X))  \rangle_{1}\;\;\;\;\text{(by (\ref{triangle4}))}\\
                 &\subseteq\langle\Omega_{\mathscr{D}}^{m+1}(F_{t_{\mathcal{V}}}(X))  \rangle_{1}\diamond \langle V \rangle_{1}\;\;\;\;\text{(by Lemma \ref{lem3.5})}.
    \end{align*}
By replacing $X$ with $F^{i}_{t_{\mathcal{V}}}(X)$ for each $1\leqslant i \leqslant n-1$,
we can get
$$\langle \Omega_{\mathscr{D}}^{m+1}(X) \rangle_{1}\subseteq\langle\Omega_{\mathscr{D}}^{m+1}(F^{n}_{t_{\mathcal{V}}}(X))  \rangle_{1}\diamond \langle V \rangle_{n}. $$
Note that $\ell\ell^{t_{\mathcal{V}}}(F^{n}_{t_{\mathcal{V}}}(X_{i}))=0$ for each $i$, we know that
$F^{n}_{t_{\mathcal{V}}}(X_{i})\in \mathfrak{F}(\mathcal{V})$ by \cite[Proposition 3.1]{zheng2020upper}.
Now by Lemma \ref{lem3.6}, we have $\Omega_{\mathscr{D}}^{m+1}(F^{n}_{t_{\mathcal{V}}}(X))=0$ in
$D^{b}(\mod\Lambda)$. Then
\begin{gather}\label{triangle6}
  \xymatrix{
 \langle \Omega_{\mathscr{D}}^{m+1}(X) \rangle_{1}\subseteq\langle V\rangle_{n}.
  }
\end{gather}
By Lemma \ref{lemwei3}, for each integer $j$, we have the following triangles
$$ \Omega_{\mathscr{D}}^{j+1}(X) \longrightarrow Q_{j}
 \longrightarrow \Omega_{\mathscr{D}}^{j}(X)\longrightarrow \Omega_{\mathscr{D}}^{j+1}(X)[1] $$
 with $Q_{j}$ projective. We can get
 \begin{align*}
   \langle \Omega_{\mathscr{D}}^{j}(X) \rangle_{1}&\subseteq\langle Q_{j} \rangle_{1}\diamond\langle\Omega_{\mathscr{D}}^{j+1}(X)[1] \rangle_{1}\\
                &\subseteq\langle\Lambda \rangle_{1}\diamond\langle\Omega_{\mathscr{D}}^{j+1}(X)\rangle_{1}.
    \end{align*}
And by Lemma \ref{lem2.5}, we have
\begin{align*}
\langle X\rangle_{1}&\subseteq\langle \Lambda \rangle_{m+1}\diamond \langle \Omega_{\mathscr{D}}^{m+1}(X) \rangle_{1}\\
                &\subseteq\langle \Lambda \rangle_{m+1}\diamond \langle V \rangle_{n}.\;\;\;\;\text{(by (\ref{triangle6}))}\\
                &\subseteq\langle \Lambda \oplus V \rangle_{m+1+n}.\;\;\;\;\text{(by Lemma \ref{lem2.5})}
    \end{align*}
Note that, for each complex $X\in D^{b}(\mod\Lambda)$,
 we have $X[p]\in \mathscr{D}^{(-\infty,-1]}(\mod\Lambda)$
 for some integer $p$,
and also note that $\langle X\rangle_{1}=\langle X[p]\rangle_{1}$. Thus, for each complex $X\in D^{b}(\mod\Lambda)$,
we always have $D^{b}(\mod\Lambda)= \langle \Lambda \oplus V \rangle_{m+1+n}.$
By Definition \ref{tri.dimenson2.1}, we have $\tridim D^{b}(\mod\Lambda)\leqslant n+m.$
\end{proof}

\begin{corollary}
Let $\Lambda$ be an artin algebra, and $\mathcal{V}$ a subset of all simple modules in $\mod\Lambda$.
If $\ell\ell^{t_{\mathcal{V}}}(\Lambda_{\Lambda}) \leqslant 2$, then
$\tridim D^{b}(\mod \Lambda)\leqslant \pd \mathcal{V}+\ell\ell^{t_{\mathcal{V}}}(\Lambda_{\Lambda})\leqslant \pd \mathcal{V}+2$.
\end{corollary}
\begin{proof}
We only need to consider the case $\pd \mathcal{V}<\infty$.
By \cite{zheng2021radicalfinite}, we know that $\Lambda$ is syzygy finite algebra, and then we know
that $\Lambda/\rad \Lambda$ has syzygy finite type. And by Theorem \ref{maintheorem}, we get
$\tridim D^{b}(\mod \Lambda)\leqslant \pd \mathcal{V}+\ell\ell^{t_{\mathcal{V}}}(\Lambda_{\Lambda}) $.
\end{proof}

\begin{corollary}
Let $\Lambda$ be an artin algebra, and $\mathcal{V}$ a subset of all simple modules in $\mod\Lambda$.
If $\gldim \Lambda <\infty$, then
$\tridim D^{b}(\mod \Lambda)\leqslant \pd \mathcal{V}+\ell\ell^{t_{\mathcal{V}}}(\Lambda_{\Lambda}) $.
\end{corollary}
\begin{proof}
Since $\gldim \Lambda <\infty$, we know that $\Lambda/\rad \Lambda$ has syzygy finite type.
And by Theorem \ref{maintheorem}, we get
$\tridim D^{b}(\mod \Lambda)\leqslant \pd \mathcal{V}+\ell\ell^{t_{\mathcal{V}}}(\Lambda_{\Lambda}) $.
\end{proof}

\begin{remark}{\rm
If we consider the algebra $\Lambda$ such that $\Lambda/\rad\Lambda$ has cosyzygy finite
 type(\cite[Definition 7.1]{rickard2019unbounded}),
 then we can get the dual results. In fact, for a module $M\in\mod \Lambda$, we use $\id M$ to denote the
 injective dimensions of $M$;
in particular, set $\id M=-1$ if $M=0$.
 For a subclass $\mathcal{B}$ of $\mod \Lambda$,
  the {\bf injective dimension} $\id\mathcal{B}$
of $\mathcal{B}$ is defined as
\begin{equation*}
\id \mathcal{B}=
\begin{cases}
\sup\{\id M\;|\; M\in \mathcal{B}\}, & \text{if} \;\; \mathcal{B}\neq \varnothing;\\
-1,&\text{if} \;\; \mathcal{B}=\varnothing.
\end{cases}
\end{equation*}
Let $\Lambda$ be an artin algebra, and $\mathcal{V}$ a subset of all simple modules in $\mod\Lambda$.
Suppose that $\Lambda/\rad \Lambda$ has finite cosyzygy type, then
$\tridim D^{b}(\mod \Lambda)\leqslant \ell\ell^{t_{\mathcal{V}}}(\Lambda_{\Lambda})+\id\mathcal{V}.$

}
\end{remark}
\begin{example}\label{ex-3.21}
{\rm (\cite[Example 3.10]{zheng2020extension})
Consider the bound quiver algebra $\Lambda=kQ/I$, where $k$ is a field and $Q$ is given by
$$\xymatrix{
&2n+1\\
{2n}&1\ar[l]^{\alpha_{2n}}\ar[u]^{\alpha_{2n+1}}\ar[r]^{\alpha_{1}} \ar[d]^{\alpha_{n+1}}
&2\ar[r]^{\alpha_{2}} &{3}\ar[r]^{\alpha_{3}}&\cdots \ar[r]^{\alpha_{n-1}}&{n}\\
&n+1\ar[r]^{\alpha_{n+2}}
&n+2\ar[r]^{\alpha_{n+3}}&n+3\ar[r]^{\alpha_{n+4}}&\cdots\ar[r]^{\alpha_{2n-1}}&2n-1
}$$
and $I$ is generated by
$\{\alpha_{i}\alpha_{i+1}\;| \;n+1\leqslant i\leqslant 2n-1\}$ with $n\geqslant 5$.

Let $\mathcal{V}:=\{S(i)\mid 2 \leqslant i \leqslant n\}$, we can see
$\pd\mathcal{V}=1$ and
$\ell\ell^{t_{\mathcal{V}}}(\Lambda_{\Lambda})=2$.

(1) Since $\LL(\Lambda)=n$ and $\gldim \Lambda=n-1$, we have
$$\tridim D^{b}(\mod \Lambda) \leqslant \min\{ \gldim \Lambda, \LL(\Lambda)-1\}=n-1$$ by Theorem \ref{thm1.1}(1)(2).

(2) Since $\gldim \Lambda<\infty$, we know
that $\Lambda/\rad\Lambda$ has syzygy finite type.
 By Theorem \ref{maintheorem}, we have
$$\tridim D^{b}(\mod \Lambda) \leqslant \pd \mathcal{V}+\ell\ell^{t_{\mathcal{V}}}(\Lambda)=1+2=3.$$
Note that, the difference of the upper bounds between $n-1$ and $3$ may be arbitrarily large.
}
\end{example}

\vspace{0.6cm}

{\bf Acknowledgements.}
The author would like to thank Professor Zhaoyong Huang for his encouragement.
This work was supported by the National Natural Science Foundation of China(Grant No. 12001508).


%
\bibliographystyle{abbrv}

\end{document}